\documentclass[12pt,a4,twoside]{amsart}
\usepackage{amssymb}
\usepackage{latexsym}

\usepackage{epsfig}
\usepackage{srcltx}
\usepackage{hyperref}
\usepackage{pdfpages}

\usepackage[hang]{subfigure}

\theoremstyle{plain}
\newtheorem{theorem}{Theorem}

\newtheorem{proposition}{Proposition}

\theoremstyle{definition}
\newtheorem{definition}{Definition}
\newtheorem{example}{Example}

\newtheorem{remark}{Remark}

\theoremstyle{plain}
\newtoks\thehProclaim
\newtheorem*{Proclaim}{\the\thehProclaim}

\theoremstyle{definition}
\newtoks{\thehRemark}
\newtheorem*{Remark}{\the\thehRemark}

\begin{document}

\title{Extremal Area of polygons sliding along curves}

\author{Dirk Siersma}

\address{Mathematisch Instituut  Universiteit Utrecht}




\begin{abstract}
In this paper we study the area function  of  polygons, where the vertices are  sliding along curves. We give geometric criteria for the critical points and determine also the Hesse matrix at those points. This is the starting point for a Morse-theoretic approach, which includes the relation with the topology of the configuration spaces. Moreover the condition for extremal area gives rise to a new type of billiard: the inner area billiard.
\end{abstract}


\maketitle

\section{Introduction}

The study of extremal positions of geometric figures has a long tradition. Well known is the {\it Isoperimetric Problem}: Determine the maximal  area of a plane figure with given perimeter, see e.g. the historical overview of
Bl\r{a}sj\" o  \cite{isoper}.

\smallskip
\noindent
In this article we focus on polygons, where each vertex slides along its given curve. For triangles this has been studied  more than 100 years ago by E. B. Wilson \cite{wilson} , which gave a geometric criterion for a triangle with maximal area.

\smallskip
\noindent
We consider arbitrary curves, which are (piecewise)  $C^{2}$ and develop a  general theory for critical polygons of the area function and their Morse theory. In this way we consider all  critical polygons and not only maxima and minima.
We formulate first results for disjoint curves, but later we also treat  intersecting curves and even polygons which have all vertices on a single curve.  As long as vertices don't coincide there is no difference.

\smallskip
\noindent
First we determine in Theorem  \ref{t:crit_va} the condition for a critical polygon: The tangent line at a vertex is parallel to the (nearest) small diagonal or two neighbouring vertices coincide. Next we compute in Proposition \ref{hessn} the Hesse matrix at a critical polygon. This matrix depends only on the vertices and on the curvature at the vertices of the critical polygon. We apply this to examples, containing lines or circles.

In section \ref{s:birth} discuss the birth and death of circles originating from a point (considered as a constant curve).
We show that generically a critical point in the original setting gives rise to two critical points in the new setting and compare the Morse indices.

In section  \ref{s:allonone} we discuss polygons, where all vertices are on a single curve. As long no vertices coincide we can use the theory of the first sections. Degenerate polygons (e.g. all vertices coincide)  are examples of critical points, which can produces non-isolated singularities.  We also discuss the `adding of a zig-zag'  and its effect on the Morse indices.

In section \ref{s:piecewise} we pay attention to the piece-wise differentiable case. Clark subdifferential replaces the usual derivative and tangent cones replace tangent lines. The case of piecewise straight lines (e.g polygons) is an important issue in computational geometry.

In section \ref{s:tsliding} we give  a short introduction to tangential sliding and the conditions for critical area in that case.

We close in section \ref{s:billiard} with a proposal for a new billiard: {\it The Inner Area Billiard.}  The billiard rules follow the conditions for critical points of the (vertex) area function:  Every vertex $P_{k+1}$ is constructed by intersecting the boundary curve of the billiard table with the ray from  vertex $P_{k-1}$ parallel to the tangent line in vertex $P_k$ . This resembles  both the usual (perimeter) billiard as the outer (area) billiard. This (as a starting point) rises several billiard type questions.

\smallskip
\noindent
 Note that area functions are affine invariants, so statements stay valid after  an affine transformation.

\smallskip
\noindent
The Morse theoretic approach has been already carried out for the signed area function on linkages with given edge
length \cite{khipan}, \cite{ks1}, \cite{ks}, \cite{panzh} and more recently in the context of the isoperimetric problem \cite{KPS}.
The case of polygons with vertices on a single ellipse is treated in \cite{Si-circle}.

This paper originated from many discussions with Gaiane Panina and George Khimshiasvili during several `Research in Residence' visits at CIRM in  Luminy.  I wish to thank both and moreover the CIRM and the Mathematical Department of Utrecht University for the good working atmosphere.

\section{Vertex Sliding of polygons}\label{s:general}

\subsection{Critical Points}

We consider a set of curves $C_1,\cdots, C_n$, embedded  in the plane.   Each curve is given by a parametrization $C_i(t_i)$. 

\begin{figure}[h]
\begin{center}
\includegraphics[height=5cm]{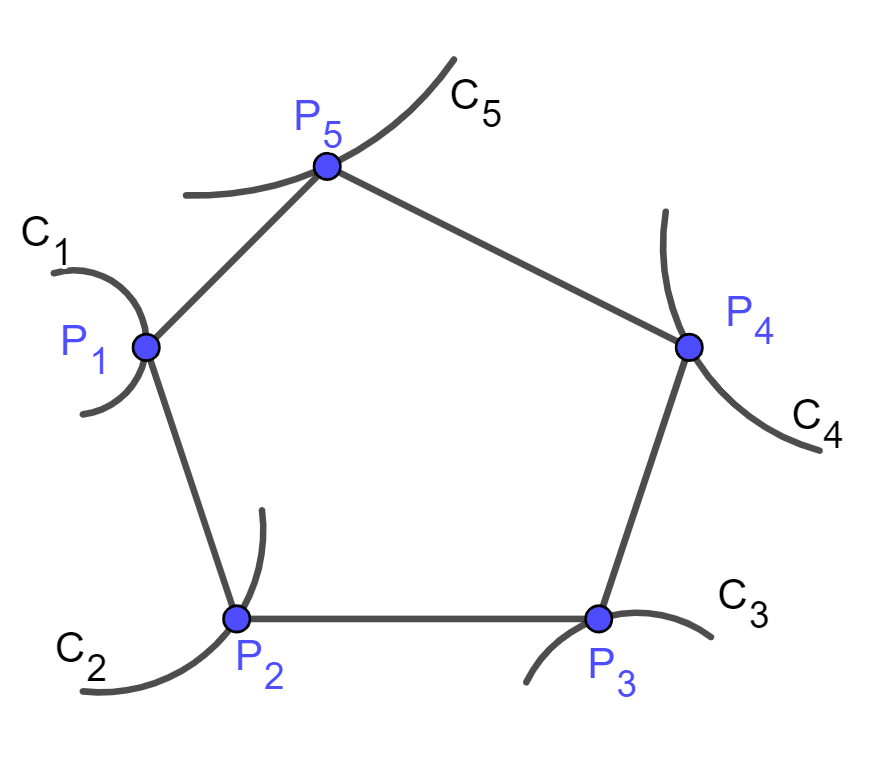}
\caption{Vertex sliding.}
\label{fig:vertexsliding}
\end{center}
\end{figure}

We denote by $C^{'}_i$ the first derivative, by $ C^{''}_i$ the second derivative, by $T_i$ the unit tangent vector, by $N_i$ the unit normal  satisfying $T_i \times N_i =1$ and by  $\kappa_i$ the curvature. 
We use the convention that we write indices modulo $n$.

\noindent
On the product of the source spaces of the curves we define 
for every set of points $P_i = C_i(t_i)$ the signed area function   $\mathcal A$ as (2-times)  the signed area of the polygon $\mathcal{P}$ with vertices $P_1,\cdots, P_n$ (in that order) with coordinates $P_i =(x_i,y_i)$ by
  $$ \mathcal A = \sum_{i=1}^n P_i \times P_{i+1} = x_1 y_2 - x_2 y_1 +  \cdots  +x_n y_1 - x_1 y_n. $$
\noindent
We have the following condition for critical points of  $\mathcal A$:

\begin{theorem} \label{t:crit_va}
Let $C_1,\cdots, C_n$ be smooth  curves in the plane.\\
$\mathcal A$ has a critical point at the polygon $P_1 \cdots P_n$ iff 
$ C_i^{'} \times ( C_{i+1} - C_{i-1}) = 0 $ for all $i$   at $(t_1,\cdots,t_n).$ 
This means:
 \begin{itemize}
\item $P_{i+1}=P_{i-1}$  or
\item $T(P_i) \parallel  P_{i-1}P_{i+1}$
\end{itemize}
In case of disjoint smooth curves only the parallel conditions apply.\\
\end{theorem}
\begin{proof}
$$ \mathcal A = C_1(t_1) \times C_2(t_2) + C_2(t_2) \times C_3(t_3) + \cdots + C_n(t_n) \times C_1(t_1). $$
Here $ \times$ denote the cross product.
The partial derivatives with respect  to $t_i$ must be 0:
$$ C_{i-1} \times C_i^{'} + C_i^{'} \times C_{i+1} = C_i^{'} \times ( C_{i+1} - C_{i-1}) = 0 .$$
Since the curves are disjoint and smooth the statement follows.
\end{proof}
\begin{figure}[h]
\begin{center}
\includegraphics[height=4cm]{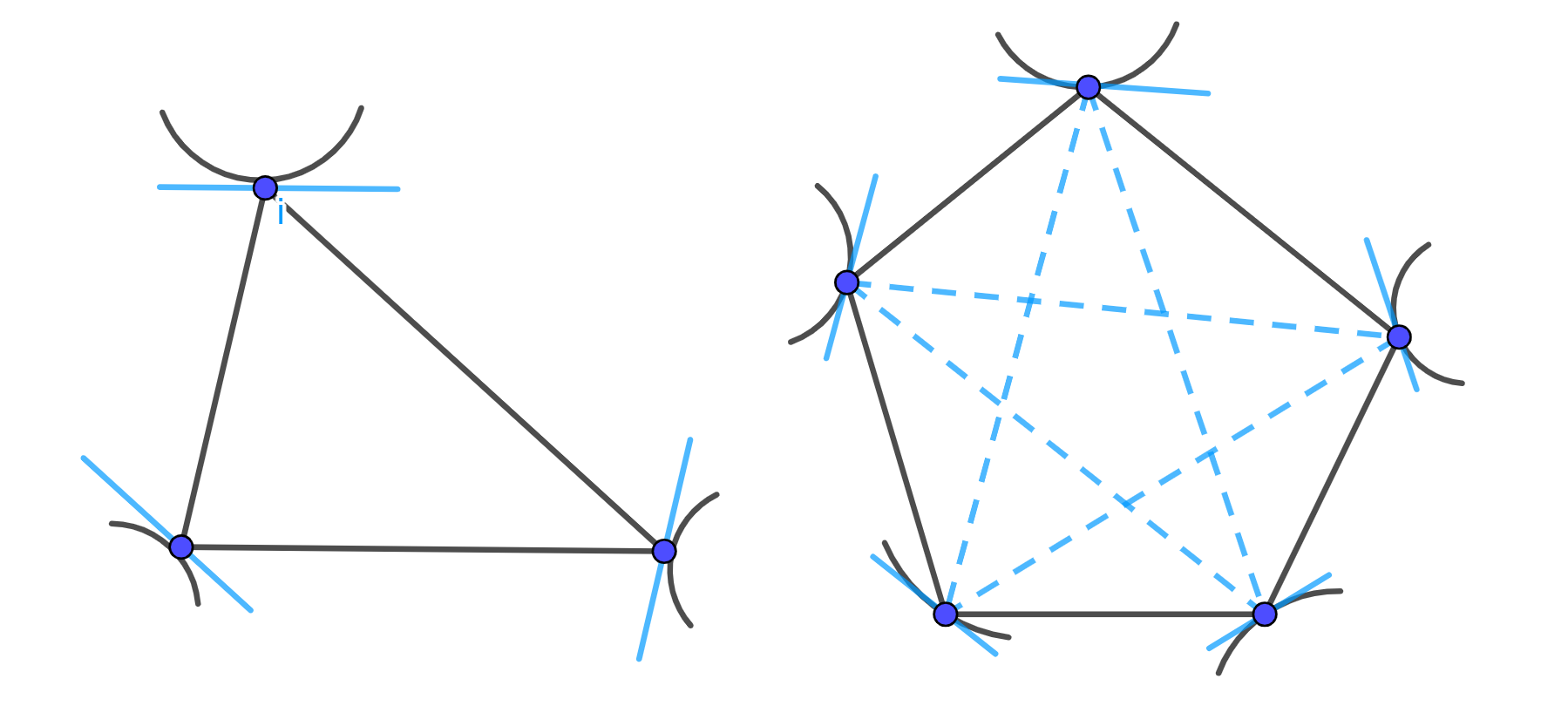}
\caption{Critical points of $\mathcal{A}$.}
\end{center}
\label{fig:vsCritical}
\end{figure}
\smallskip

\begin{remark}
{\bf a.}
If curves $C_{i-1}$ and $C_{i+1}$ intersect, then the intersection point $P_{i-1}=P_{i+1}$ together with the remaining parallel conditions define critical points. 
When the curves intersect transversally the effect is not significant. If the curves are tangent  and $\mathcal A$  is critical with $P_{i-1}=P_{i+1}$ then as a consequence the points $P_{i-2}, P_{i}, P_{i+2}$ must be collinear.

\noindent
Examples $n=3$: A transversal intersection of two curves will never occur in a critical triangle; unless
all 3 curves intersect in that point. But if the two curves are tangent in $P_1=P_3$  and $P_2$ is an intersection point of the tangent line with $C_2$ then the `triangle' $P_1P_2P_3$ is a critical point. See Figure \ref{fig:3singular}.

\begin{figure}
	\centering
		\includegraphics[width=0.80\textwidth]{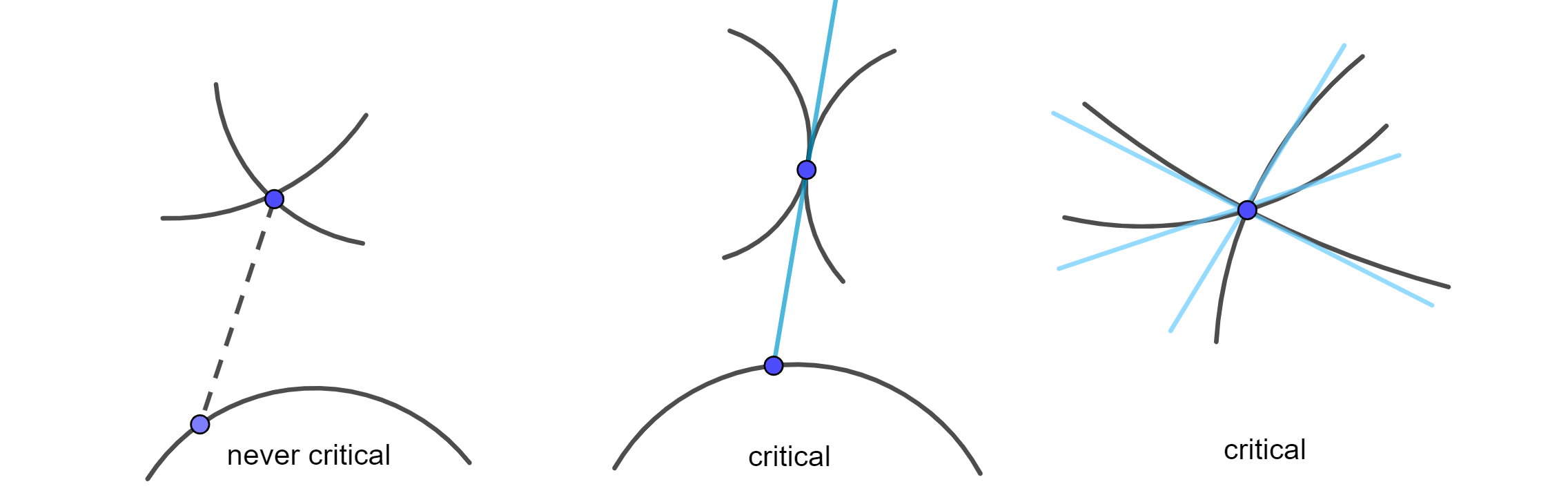}
		\caption{Singular triangles.}
	\label{fig:3singular}
\end{figure}

\smallskip
\noindent
{\bf b.} A special case is: one or more curves coincide. We will meet this in section \ref{s:allonone}.

\smallskip

\noindent
{\bf c.} It is also possible to apply the theorem in cases that one of the curves is a point. We are just left with the other partial derivative conditions.  See \ref{ss:pointlike} .
\end{remark}

\medskip
\subsection{The Hessian}
Critical points are determined by first order information about  the curves (tangent lines).
Next we focus on the second order information.

\begin{proposition}\label{hessn}

The Hesse matrix of  $\mathcal A$  is corner tridiagonal
$$     H =        \begin{pmatrix}
               a_1  &  b_1   &   0  &  \cdots & 0 &  b_n \\
              b_1  &  a_2   &  b_2 & \cdots & 0&   0\\
						0 & b_2 & \ a_3& \cdots & \cdots  & \cdots\\
						\dots & \cdots & \cdots  & \cdots &  \cdots & \dots \\
\cdots & \cdots  & \cdots &  a_{n-2} & b_{n-2} &  0 \\
						0 & 0 &  \cdots & b_{n-2} &  a_{n-1} & b_{n-1} \\
               b_n  & 0  & \cdots & 0 &  b_{n-1}   & a_n\\
             \end{pmatrix}
						$$
where $a_i =  C_i^{"} \times (C_{i+1}- C_{i-1}) $ and $b_i= C_i^{'} \times C_{i+1}^{'}$. 

\end{proposition}

\noindent
	Note that the matrix elements are $0$ as soon as $|i-j| > 1$. {\it Each of the entries are geometric.}  If we have parametrization via arc length,
	then 
\begin{center}  $b_i = T_i \times T_{i+1} = \sin \alpha_i$  
	 and 
	 $a_i  =C^{"}_i \times   P_{i-1}P_{i+1} = P_iM_i \times  P_{i-1}P_{i+1} $ , \end{center}
where $\alpha_i = \angle(T_i,T_{i+1})$  and $M_i$ is the center of curvature of $C_i $ at the point $P_i$. The two vectors in the second product are orthogonal in a critical point.  Moreover, if we give the tangent line at $P_i$ to $C_i$ the  same orientation as  $P_{1-1}P_{i+1}$ then there is the following {\em sign rule: }  $a_i > 0 $ if  $M_i$ is on the left side of the tangent line and $a_i < 0$ if $M_i$ is on the right side.
This description does not depend on the orientation of the curve $C_i$.  Also  $a_i = \kappa N_i \times P_{i-1}P_{i+1} = 
 - \kappa_i \epsilon_i l_i$, where $l_i = | C_{i+1}-C_{i-1}|$ and $\epsilon_i=+1$  depending on the sign of $a_i$.\\
NB.  The sign of $\kappa$ and the vector $N$ depend on the orientation of the curve. If one changes orientation then we get the opposite sign.   If we change orientation of one or more curves the Hesse matrix will change by a coordinate transformation of a quadratic form. The sign of the determinant, index and signature will not change. 

\medskip
\noindent
NB. In section \ref{ss:hess} we comment on the index of the critical point and how this depends on the positions of the centers of curvature.
 						
\medskip
\noindent
NB. A  symmetric  and tridiagonal matrix  (with corners)  is quite common in circular systems with neighbouring point interaction.

\medskip
\noindent
$\mathcal A$ is generically  a Morse function. This can already be arranged by translations only:
\begin{proposition}
Given $n$ vectors $a_1,\cdots,a_n$, which span the 2-dimensional plane. 
\begin{itemize}
\item
Then $\mathcal A$, defined on  the translated curves $C_1 + s_1 a_1, \cdots , C_n +s_na_n$  is Morse for almost every parameter value $s =(s_1,\cdots,s_n)$.
\item
If all curves are compact and $\mathcal A$ is Morse for $s=0$ then there is a  $\rho > 0$ such that $\mathcal A$ is also Morse for all $|s|< \rho$.
\end{itemize}
\end{proposition}
\begin{proof}
This follows from the stability and parametric transversality theorem  (\cite{GP}).

\end{proof}

\subsection{Higher order approximation}
If the critical point is degenerate (non-Morse) the higher order terms become important.
We determine here the third  order terms in the Taylor series.

\smallskip
\noindent
We still consider parametrization by arc length and fix notations: If  $T$ is the unit tangent vector, then the unit normal vector is defined by $T \times N =1 $. In this case $T^{'} = \kappa N$ and $N ^{'}= - \kappa T$.
The terms of order $3$ are: 
$$ \frac{1}{3!} \big( -  \sum_{i=1}^n \epsilon_i  l_i  \dot \kappa_i\; t_i^3  - 3 \sum_{i=1}^n \kappa_i \cos \alpha_i \; (t_i^2 t_{i+1} - t_i t_{i+1}^2) \big)$$

\smallskip
\noindent
This follows from the computation of the 3rd order derivatives:
$$ a_{i i i} = C_i^{'''} \times (C_{i+1} - C_{i-1}) = (\dot \kappa_i N_i - \kappa_i^2 T_i) \times \epsilon _i l_i T_i = - \epsilon_i l_i \dot \kappa_i  $$
$$
a_{i+1 i i} =  C_i ^{''}\times  C_{i+1}^{'} = \kappa_i N_i \times T_{i+1} = - \kappa_i \cos \alpha_i
$$
$$
a_{i+1  i+1  i} = C_i ^{'}\times  C_{i+1}^{''}= \kappa_i T_i \times N_{i+1} = - \kappa_i \cos \alpha_i
$$
$$
a_{i j k} = 0 \; \; \mbox{\rm in all other cases}
$$

\section{Special cases of vertex sliding}
\noindent
In this section we present several cases, where the curves are lines or circles.

\subsection{Sliding along straight lines}

The extremal point conditions for $\mathcal A$  depend only on the tangent lines. It turns out that the study of lines
is an important ingredient in the understanding of more general curves.
We first consider the case of 3 lines (with has a surprizing nice answer).

\begin{proposition}
In the case of three lines (not through one point) there are global coordinates such that the area function  is given by $$ \mathcal A  =  3 (t_1t_2 + t_2 t_3 + t_3 t_1 + \tfrac{1}{4}) \mathcal A  (B_1B_2B_3).$$
where $B_1B_2B_3$ is  the triangle formed by the intersection points of the lines.
$\mathcal A$ has exactly one critical point, which is Morse and has index 2. 
\end{proposition}

\begin{figure}[ht]
\begin{minipage}[b]{0.45\linewidth}
\centering
\includegraphics[width=\textwidth]{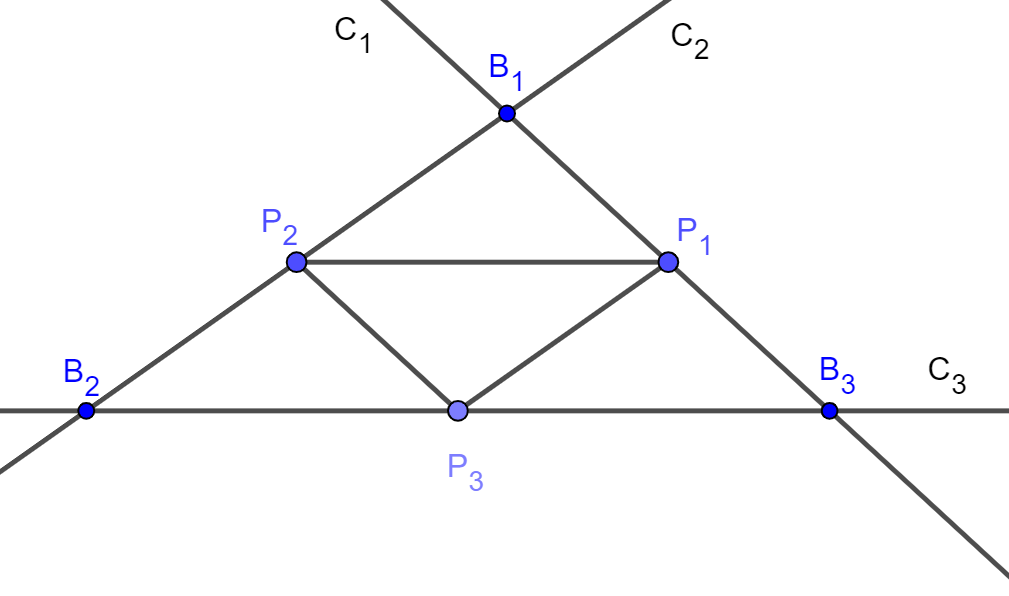}
\caption{Three lines.}
\label{fig:3lines}
\end{minipage}
\hspace{0.5cm}
\begin{minipage}[b]{0.45\linewidth}
\centering
\includegraphics[width=1.0\textwidth]{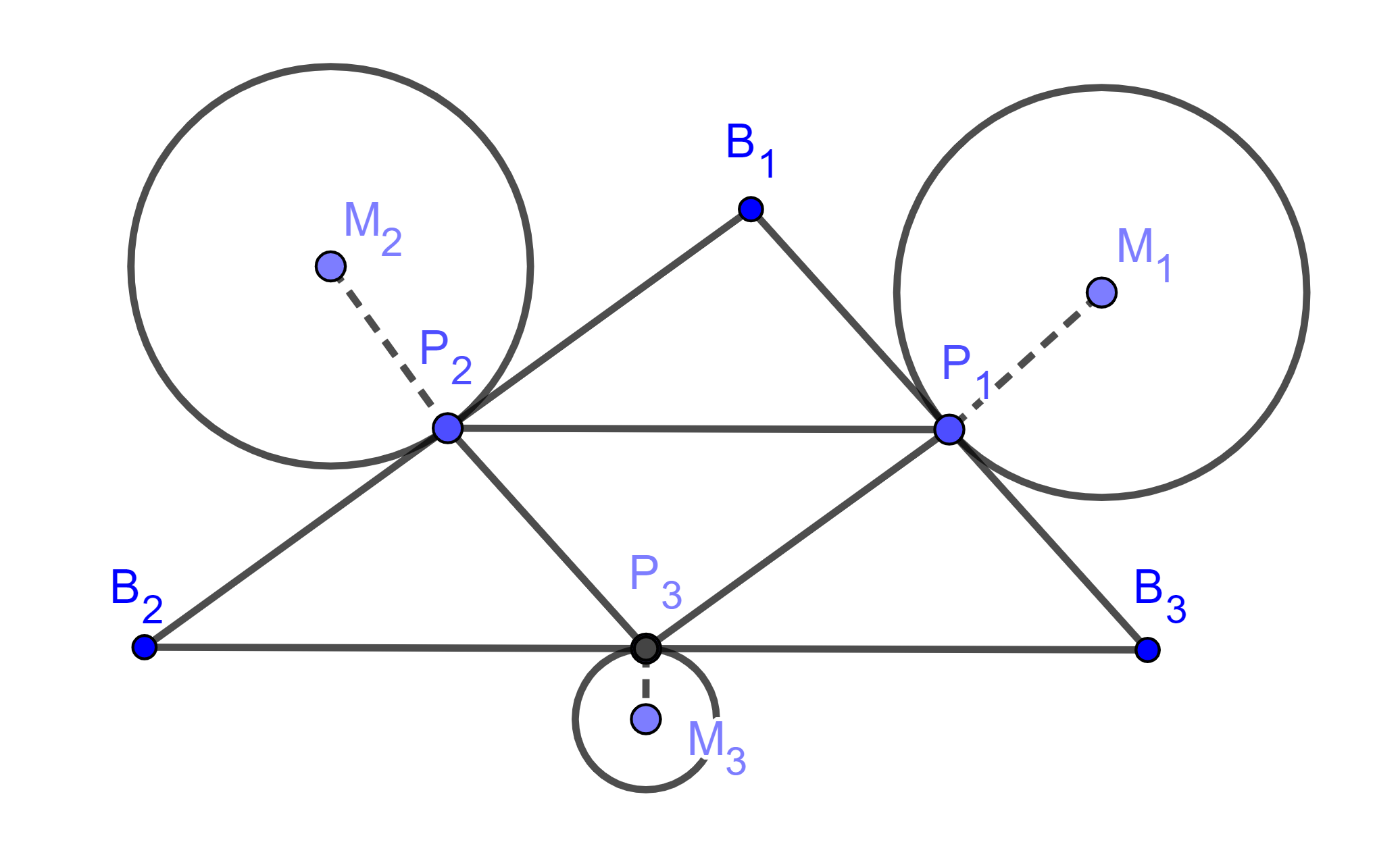}
\caption{Three circles.}
\label{fig:3circleHessian}
\end{minipage}
\end{figure}

\begin{proof}
Let $C_i$ and $C_{i+1}$ intersect in $B_i$. Use parametrization $C_i(t_i) = t_i B_{i-1} + (1-t_i) B_i$ and perform the computation. After a translation in the coordinates one gets the formula in the theorem.
\end{proof}

\medskip
\begin{remark}
In case of 4 lines the parallel conditions imply that opposite sides a parallel (and enclose a parallelogram). In that case there are even infinitely many solutions (non-isolated singular 4-gons).  The study of $n$ lines has its own interest, which will be discussed in a future paper.  
\end{remark}
\subsection{Sliding along circles; Hessian and centers of curvature} \label{ss:hess}

The critical points and their Hesse matrices depend only on the 2-jets of the curves. For the local study of critical points we can therefore replace the curves by circles, centered in  the center of curvature and with radius equal to the radius of curvature. 
 
\subsubsection{Three circles in arbitrary position}
We consider the case of 3 circles with different radii $r_i $ and centers $M_i$. If $\mathcal A$  is critical then the 3 tangent lines at the vertices of the sliding triangle are parallel to the opposite side of the triangle (see Figure \ref{fig:3circleHessian}). Use clockwise orientation of the circles and their tangent lines
.

\begin{proposition}
There exists coordinates such that the Hesse matrix at a critical point  of $\mathcal A$ is:
\begin{equation} \label{eq:hes3}
         \begin{pmatrix}
               -m_1 & s   &  s \\
             s  &  -m_2 & s \\
             s & s & -m_3\\
             \end{pmatrix}
\end{equation}
\noindent
where $l_i$ is the length of the  $i^{th}$ edge of the triangle $P_1P_2P_s$ , $m_i = \kappa_i\epsilon_i l_i^3$  and $s=\frac{1}{2} \mathcal A (B_1,B_2,B_3)$. 

\end{proposition}
\begin{proof}
Use parametrization such that $||C_i^{'}(t_i)||= l_i$ .
Next insert this in the Hesse matrix in Proposition \ref{hessn}.
\end{proof}

The determinant of this matrix is:  \label{eq:det3}
\begin{equation} -m_1 m_2 m_3 + ( m_1 + m_2 + m_3) s^2  + 2 s^3 
\end{equation}
and the eigenvalue equation is:
\begin{equation} \label{eq:ev3}
 -(\lambda +m_1) (\lambda+ m_2)(\lambda + m_3) + ( 3 \lambda + m_1 + m_2 + m_3) s^2  + 2 s^3 = 0
 \end{equation}

\noindent
{\bf Discussion:}
 The index of the Hessian  depends on a relation between all three radii of curvature.  The Hesse matrix is determined by the positions of the 6 points;  namely $P_1,P_2,P_3$ and the centers of curvatures $M_1,M_2,M_3$. A question is: Is there a geometric criterion in terms of these points, which gives the index or tells when the critical point is non-Morse ?  Other questions are:
What happens for big and small radii? What if the 3 points lie on  the inscribed circle of the triangle defined by the tangent lines ? 

\smallskip
\noindent
 The above  matrix and formula for the determinant imply already some corollaries:
 \begin{itemize}
 \item If all $|m_i|  > > 0$ then the term $- m_1 m_2 m_3$ is dominant in the Hessian determinant: So for circles with very small radius this has de effect on the eigenvalues. They are determined  by the signs of $\kappa_i$. 
 \item If all $m_i=0$ then we have saddles (see the section on 3 lines); and this is still the case for very small values of the curvature.
 \end{itemize}
 
\subsubsection{Bifurcations}
 What can be said about the bifurcation theory for three arbitrary circles ? 
 
 We start with a triangle $P_1P_2P_3$. We will use the parallels $P_{i-1}P_{i+1}$ through $P_i$ as future tangent line to the circles. The centers $M_1,M_2,M_3$ are situated at distances $r_i$ on the perpendiculars at $P_i$ to these lines.  We consider the three circles $(M_i,r_i)$. They are indeed tangent to the mentioned lines.
 Note that for all values of $r_i$ the polygon $P_1P_2P_3$ is a critical polygon. 

\smallskip
\noindent Consider as first example  the following 1-parameter family of circles:
 Fix  $r_1$ and $r_2$ and let $r_3$ vary.  The vanishing of the Hessian determinant \eqref{eq:det3} gives us 
 (under the condition  $m_1 m_2  \neq s^2$)  exactly one bifurcation value $m_3^b$ for $m_3$.  To be more precise:
 $$ m_3 =      \frac{m_1 s^2 + m_2 s^2 + 2 s^3 }{m_1m_2-s^2} $$
 For this value $m_3^b$  the Hessian determinant (evaluated for the polygon $\mathcal P =P_1P_2P_3$ changes sign.  What happens? A computation with the 3-jet of $\mathcal A$ shows, that after a coordinate transform we get the family:
 $ - m_1 x^2 - m_2  y^2 + \omega z^2 +  z^3$, where $\omega$ measures the (signed)  difference $m_3- m_3^b$. This means, that $\mathcal A$  has a critical point of type $A_2$ for that value. In the family a second critical polygon meets our critical polygon at the bifurcation value and moves away after that, while both change to the opposite index.

 \medskip
\noindent
 Next we consider the family where $m_1=m_2=m_3=m$. In that case formula \eqref{eq:det3} for the Hessian determinant becomes:
 \begin{equation} \label{eq:det3eq}
   - m^3  + 3 m s^2 + 2 s^3
 \end{equation}
 This happens e.g. when in the above description  $\mathcal P$ is equilateral and all radii equal: $r_1=r_2=r_3=r$.
The Hessian determinant is zero in 2 cases:
$$  m =2s    \;  , \;   m =- s  \mbox{\rm (double root)}  $$
At these two bifurcation values, the first corresponds to the case that all three circles coincide with the inscribed circle of the triangle  and  $\mathcal A$ has a non-isolated singularity (of type $A_{\infty}$). The second to a singularity of corank 2 (type $D$). \\
The eigenvalue equation  becomes:
$$ - (\lambda + m  - 2s)(\lambda + m+ s^2 = 0 $$
This determines all the Morse indices of $\mathcal A$ at the critical polygon $\mathcal P$.

\subsubsection{Four circles in arbitrary position}
We consider the case of 4 circles with different radii $r_i $ and centers $M_i$. If $\mathcal A$  is critical we have that the  tangent lines at the vertices of the sliding 4-gon are parallel in pairs to the two diagonals of the quadrilateral.  Consider such a situation (see Figure \ref{fig:4critical}).
\begin{figure}[h]
	\centering
		\includegraphics[width=0.50\textwidth]{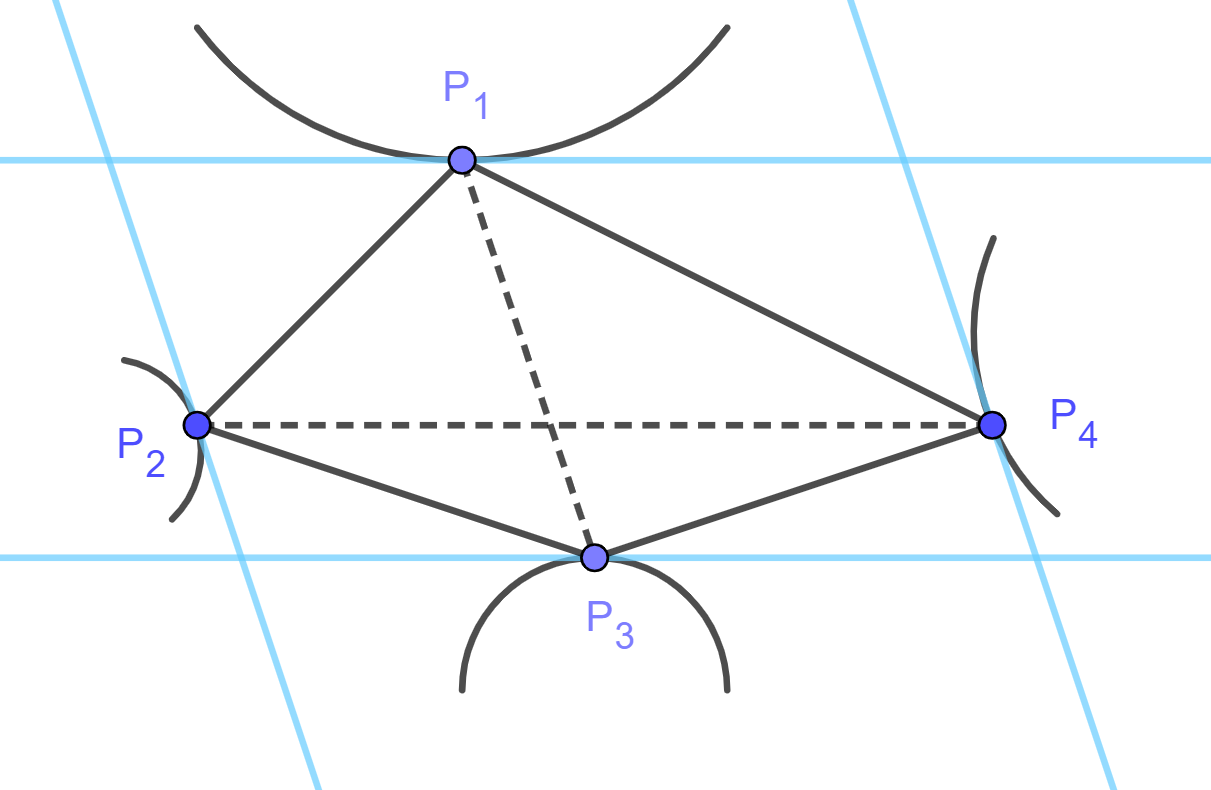}
	\caption{Critical with 4 curves.}
	\label{fig:4critical}
\end{figure}

\begin{proposition}
There exists coordinates  such that the Hesse matrix is:
$$             \begin{pmatrix}
               -m_1 & s   & 0 & s\\
             s  &  -m_2 & s & 0 \\
             0 & s & -m_3 & s \\
						 s & 0 & s & -m_4\\
             \end{pmatrix}
$$
where  $m_i = -\kappa_i  \|P_{i-1}P_{i+1}\|$  and $s= \sin \alpha_{i,i+1}$  where $\alpha_{i,i+1} = \angle(T_i,T_{i+1}).$  
\end{proposition}

\begin{proof}
Use parametrization of the circles (or curves) by arc length.
It is clear that $\alpha_{12}=\alpha_{34} = \pi - \alpha_{23} = \pi - \alpha_{41}$.
Next insert this in the matrix of Proposition \ref{hessn} .
\end{proof}
The determinant of this matrix  is
$$ m_1 m_2 m_3 m_4 - (m_1 m_2 + m_2 m_3 + m_3 m_4 + m_4 m_1) s^2 $$
For the eigenvalues: replace $m_i$ by $\lambda + m_i$.\\
Specializing to all $m_i$ are equal we get the eigenvalue equation
$$ (m+\lambda)^2 (m + \lambda -2s)(m+\lambda +2s)= 0 $$
\subsection{Computations with circles} \label{ss:circles}
Let $M_1,\cdots,M_n$ be the centers of the circles and $r_i$ the corresponding radii.
A point $P_i$ on circle $C_i$ is given by $OM_i + r_i(\cos \alpha_i,\sin \alpha_i)$ and 
$$ \mathcal A = \sum_{i=1}^{n} (OM_i + r_i(\cos \alpha_i,\sin \alpha_i)) \times (OM_{i+1} + r_{i+1}(\cos \alpha_{i+1},\sin \alpha_{i+1})).$$
The usual questions are now:
Determine {\it all} critical points and their index and to test this with the topology of the n-torus. Is  $\mathcal{A}$ a perfect Morse function ?  Because of the complexity  of the computation it was  not possible to find solutions in the general case.
This seems also to be the case if we use Lagrange multipliers. In certain explicit examples  with fixed parameter one can use computer algebra systems for solving.

\smallskip
\noindent
We look now next to some special cases, where we  try to say more. 

\subsection{Concentric circles}

After choosing the common center M  as origin we can use vector notation.
A point $P_i$ on $C_i$ determines the vector $p_i$.
In this case the criterion for critical point reads as follows:
The vector $p_i$ is orthogonal to $p_{i+1}-p_{i-1}$ and this is equivalent to the equality of inner
products
$$p_1 \cdot p_2 = p_2 \cdot p_3 = \cdots = p_n \cdot p_1$$
This gives an equivalent trigonometric criterion in terms of angle-coordinates  with the radii as parameters.\\
The geometric criterion gives in low dimensional cases: (Figure \ref{fig:orthocenter})\\
n=3: The center O is the orthocenter of the triangle $P_1P_2P_3$, \\
n=4: The two diagonals are orthogonal and intersect in the center O. \\
\begin{figure}[h]
	\centering
		\includegraphics[width=0.60\textwidth]{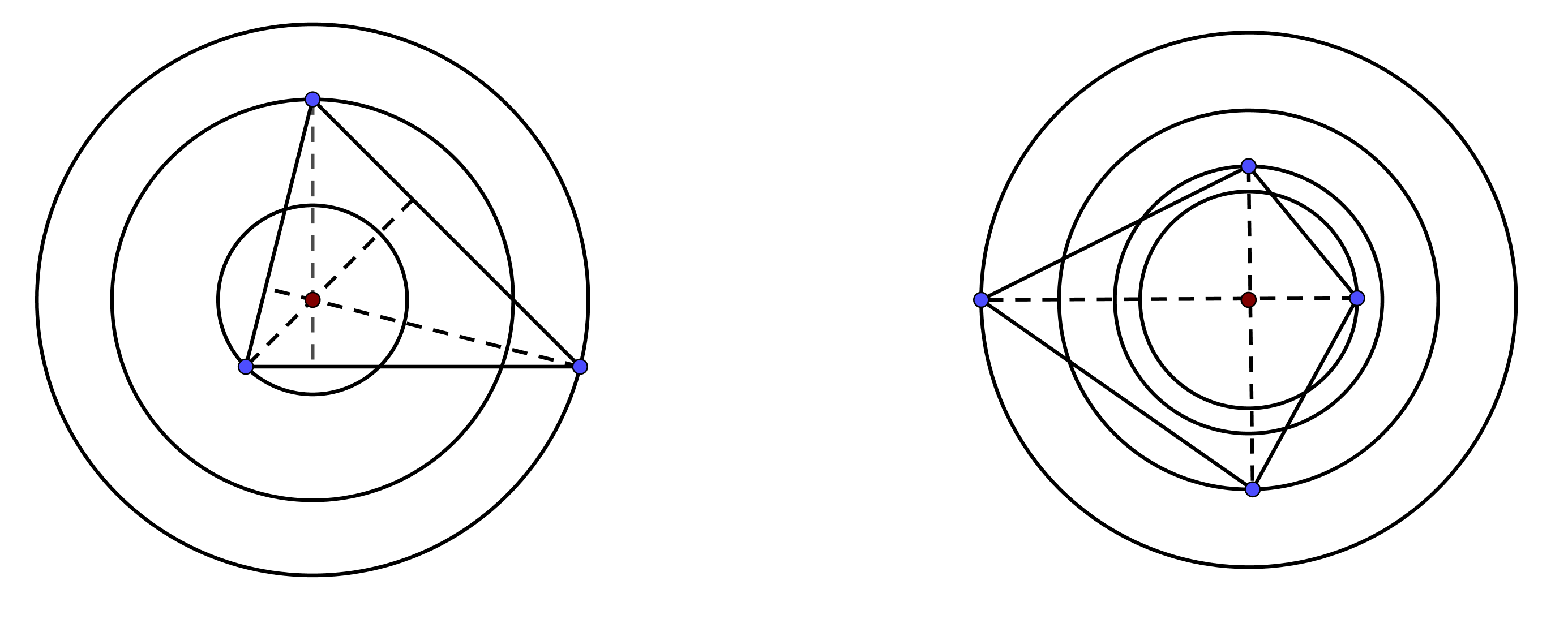}
	\caption{Geometric criterion concentric circles.}
	\label{fig:orthocenter}
\end{figure}
In the case of concentric circles we have a rotation symmetry. We will use the reduced configuration space $(S^1)^{n-1}$. In the (full) configuration space the critical points will appear as product with a circle. The same reduced configuration space occurs as configuration space of $n-1$ concentric circles and a point.

\subsubsection{Three concentric circles}

\begin{proposition}\label{p:3cc}
For 3 concentric circles (with not all radii equal) the area function $ \mathcal A$ is a perfect Morse function,
i.e. has 4 non-degenerate  critical points (1 maximum, 2 saddles  and 1 minimum).
\end{proposition}
\begin{proof}
The paper \cite{ks}  studied open n-arms, including a criterion for the critical points of the area function. In the case of 3-arm there is the following relation between the area of arms and the area of a triangle with 3 points on concentric circles:
       $$\mathcal A (a_1,a_2,a_3) =  - \mathcal A_{arm} (a_1,-a_2,a_3) $$
			As a corollary: The two critical point theories (for 3-arms and for 3 concentric circles) are equivalent. 
			Proposition \ref{p:3cc} follows now Theorem 2.1 from \cite{ks}, more especially from the detailed computation in \cite{ks1}.
\end{proof}
\noindent
NB. The criterion for critical 3-arm is (cf Theorem 1.1 of \cite{ks}) 
the diacyclic situation (all vertices of the arm are on a circle and the center of the circle is the midpoint of the endpoint vector of the arm; while in the other case the origin of the vectors is equal to the orthocenter of the triangle spanned by the endpoints of the 3 vectors.  See Figure \ref{fig:ArmVersusConc}.
\begin{figure}[h]
	\centering
		\includegraphics[width=0.60\textwidth]{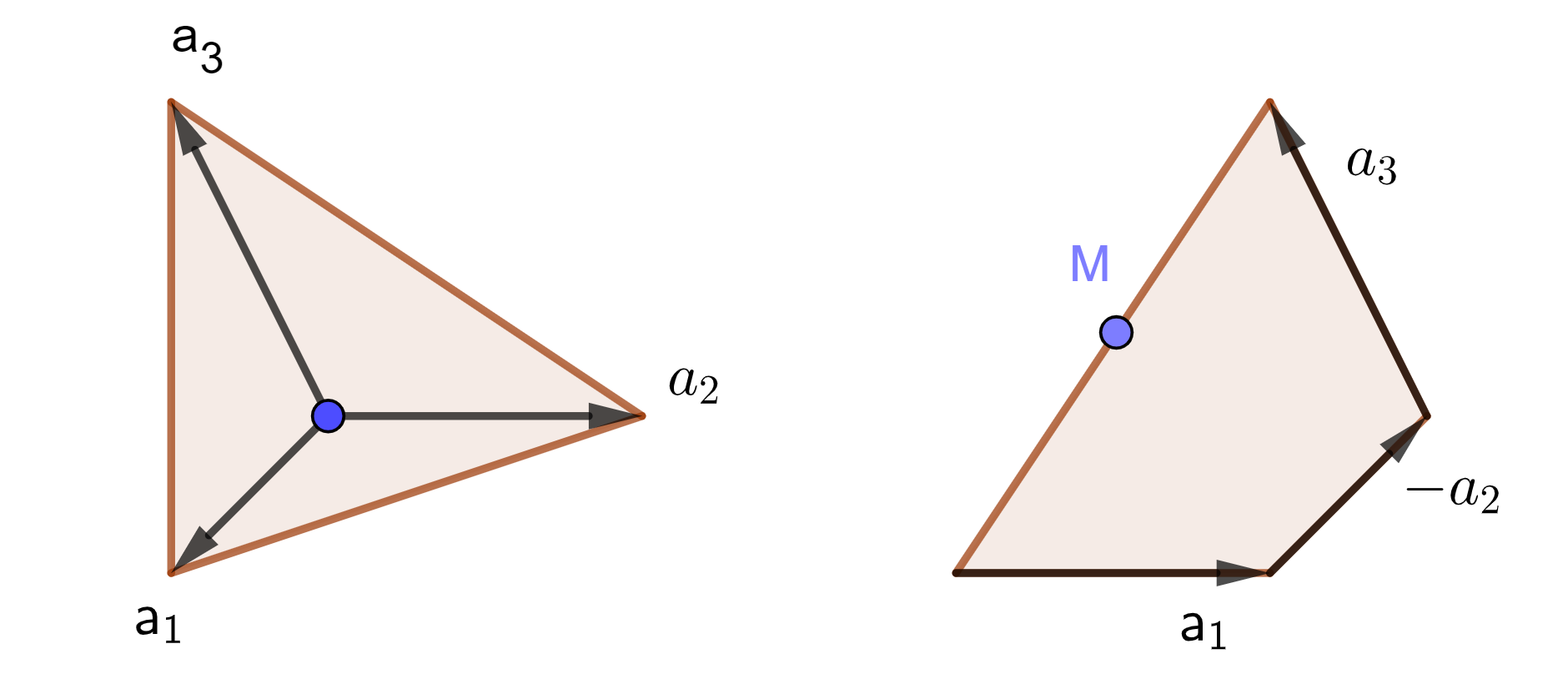}
	\caption{Concentric circles versus Arms.}
	\label{fig:ArmVersusConc}
\end{figure}

\subsubsection{Four concentric circles}

The general criterion for critical point specializes to:\\
A 4-gon with vertices on 4 concentric circles has a critical point if and only if  the diagonals are orthogonal and intersect in the center O.

\begin{proposition}
In case of  four concentric circles with $r_1 \ne r_3 $ and $r_2 \ne r_4$:
\begin{itemize}
\item $\mathcal A$ has precisely 8 critical points on the 3-torus
\item  all critical points are Morse.
\end{itemize}
As a consequence:  $\mathcal A$ is a perfect  Morse function.
\end{proposition}
\begin{proof}
Consider the following construction:\\
 Start with at any point $P_1$ on $C_1$. Take a line $l$ from that point to the center. This line has 2 intersection points $P_3^{\pm}$ with $C_3$. Take next a line through the center orthogonal to $l$ and intersect with $C_2$ and $C_4$. Altogether we have 8 possibilities $P_1 P_2^{\pm}P_3^{\pm}P^{\pm}_4$
Compare the right hand side of  Figure \ref{fig:orthocenter}.\\

\noindent
Next we compute the Hesse matrix in the critical points by using the formula in  (\ref{ss:circles}). We fix $\alpha_1=0$. 
Critical points now occur when $\alpha_2= \pm \frac{\pi}{2}, \alpha _3 = 0$ or $\pi , \alpha_4 = \pm \frac{\pi}{2}$. The Hesse matrix is as follows:
$$
\begin{pmatrix}
 r_4 r_1 - r_1 r_2  & - r_1 r_2 & 0 \\
 - r_1 r_2 & - r_1 r_2 + r_2 r_3  &  r_2 r_3\\
 0 & r_2 r_3 & - r_4 r_3+r_2 r_3 \\
\end{pmatrix}
$$
We allow also negative values of the radii, in that way we can deal with all the 8 stationary  polygons together. We require $r_1 > 0$.
The Hessian determinant is:
$$ - r_1 r_2 r_3 r_4 (r_4- r_2)(r_3 -r_1) $$
It follows that as soon as the determinant is non-zero we have 8 critical points which are all of Morse type.

\end{proof}
\section{Birth and death}\label{s:birth}
\subsection{About point-like curves} \label{ss:pointlike}
It is also possible to apply theorem \ref{t:crit_va} in cases that some  of the curves are points. We are just left with the partial derivative conditions for the remaining curves.

\medskip
\noindent
Let $C_i = P_i$ (constant).
The effect on the Hesse matrix is  that the entries $b_{i-1},a_i,b_i $ become all $0$. 
If we omit the constant variable $t_i$ the reduced matrix becomes 'tridiagonal without corners' (after the shift in numbering $i \to n$).
The seize of the Hesse matrix is reduced by the number of the point-like curves. 
$$
 \begin{pmatrix}
               a_1  &  b_1   &   0  &   0 & 0 &  b_6\\
              b_1  &  a_2   &  0 & 0  &   0 & 0\\
						0 & 0 &  0 & 0  & 0 & 0\\
0 & 0 &  0 &  a_4 & b_4 & 0\\
               0  & 0  &  0 &  b_4  & a_5 & b_5\\
 b_6  & 0  &  0 &  0 & b_5 & a_6
             \end{pmatrix}
\; \; ;  \; \;      \begin{pmatrix}
               a_1  &  b_1   &   0  &   0 & b_6 \\
              b_1  &  a_2    &   0 & 0 & 0 \\			
0 & 0 &  a_4 & b_4 & 0 \\
               0  & 0  &  b_4  & a_5 & b_5  \\
b_6 & 0 & 0 & b_5 & a_6
             \end{pmatrix}
$$
In case of two or more (pairwise) non-neighbouring points the polygon splits  into sub chains with fixed endpoints. The conditions separate the variables over the chains. The Hesse matrix becomes a block matrix with tri-diagonal blocks.

\medskip
\noindent
We mention the following sign change rule for the index:\\

\noindent
{\bf Sylvester Rule:}
{\it Let $H$ be a symmetric matrix of size $n$.  Let $H_k$ denotes the $k \times k$ submatrix consisting of the first $k$ rows and 
columns.
We consider the sequence: 
\begin{equation}\label{eq:sylv}1, \det  H_1, \det H_2 , \cdots , \det H_{n-1} , \det H_n .
\end{equation}
Under the assumption that $H_k$ is non-singular for all $k$  the index of the 
symmetric matrix $H$ is equal to the number of sign changes in the sequence.
}

\vspace{0.4cm}
\noindent
We copied this statement form the paper \cite{SV}. It is in fact a consequence of the Jacobi-Sylvester signature rule. We refer to   \cite{GR}, which contains a historical description.
\smallskip

\begin{remark}
{\em What to do if some $\det H_k = 0$ for some $k < n $ ?} We will use that for a non-degenerate matrix the index does not change under small perturbations.
Let $H[\epsilon] = H + \epsilon I$.  For a proper choice of $\epsilon$ this will not change the index of $H$ and also not of those $H_k$ where $\det H_k \ne 0$.
We can now compute the index of $H$ by counting the sign changes of $\det H[\epsilon]_k$. This argument (supplied by Van der Kallen) will be useful at several places in this paper.
\end{remark}
\subsection{The birth of tangential circles}
If we have some constant curves, let small circles grow at those points with well-chosen tangent directions and consider the effect.
Compare the left part of Fig \ref{fig:birth}.

\begin{proposition}\label{p:birth1}
Let a subset of the curves  $C_1, \cdots , C_{n}$  be constant (point curves).  Consider  a critical polygon $\mathcal P$.  Replace some (or all)  point-curves $P_i$ by  circles, such that the tangent line  in $P_i$ is parallel to $P_{i-1}P_{i+1}$. Then $\mathcal P$ is also a critical polygon for the 
updated set of curves.\\
If $\mathcal P$   is Morse for the original curves, then for small enough   radii, $\mathcal A$ is also Morse for the updated curves.  The Morse index increases with $0$ or  $1$ for each new circle, depending on position of the circle with respect to the tangent line.
\end{proposition}
\begin{proof}
Consider the small diagonals of the critical polygon $\mathcal P$ . Their directions determine also the tangent directions for a critical point is the updated problem. As soon if we replace a constant curve $P_i$  by a curve through $P_i$ with tangent direction parallel to $P_{i-1}P_{i+1}$  we satisfy the parallel conditions for the updated curves.
For the Morse theory: Consider the Hesse matrix in Proposition \ref{hessn}. The original problem corresponds to a submatrix where the $i^{th}$ rows and columns have been deleted for every born circle. (Note that we use here the remark that the second derivative with respect to $t_{i-1}t_{i+1}$ is $0$). It's determinant is by assumption non-zero. 

\smallskip
\noindent
We intend to use Sylvester's rule for the statement about indices. We assume $i=n$.
Let $C_n[r] = O[r] + r(\cos t_n, \sin t_n)$ be the circle with radius $r \ne 0$, which is tangent at  $P_n$ to the line through $P_n$, which is paralel to $P_{n-1}P_1$.
We allow $r < 0$ in order to describe circles at both sides of this line. Note that:\\
$C_n[r]^{'}= r (\sin t_n , \cos t_n)$  and  $C_n[r]^{''}= -r (\cos t_n, \sin t_n) $.  Consider next the Hesse matrix $H[r]$, with the entries:
$a_n[r] = r a_n \;  ,\;  b_n[r] = r b_n \;  , \: b_{n-1}[r] = r b_{n-1}$  where $a_n, b_n, b_{n-1}$ are the values for $r=1$.  All the other $a_i, b_i$ do not depend on $r$.
An elementary determinant computation shows:
\begin{equation}\label{eq:detH}
 \det H[r] = r a_n \det H_{n-1} +  r^2  K \; \; ; \;  \; (\mbox{ K \rm{ a constant}} )
\end{equation}
Use now Sylvester's rule. We have the assumption $\det H_{n-1} \ne 0$. It follows that for $|r|\ne 0$ small enough  $\det H[r] \ne 0$ and the sign change between the determinants is determined by the sign of $r a_n$. This increases the Morse index with 0 or 1. \\
This reasoning can be repeated for the other point curves.
\end{proof}

\begin{figure}[h]
	\centering
		\includegraphics[width=0.40\textwidth]{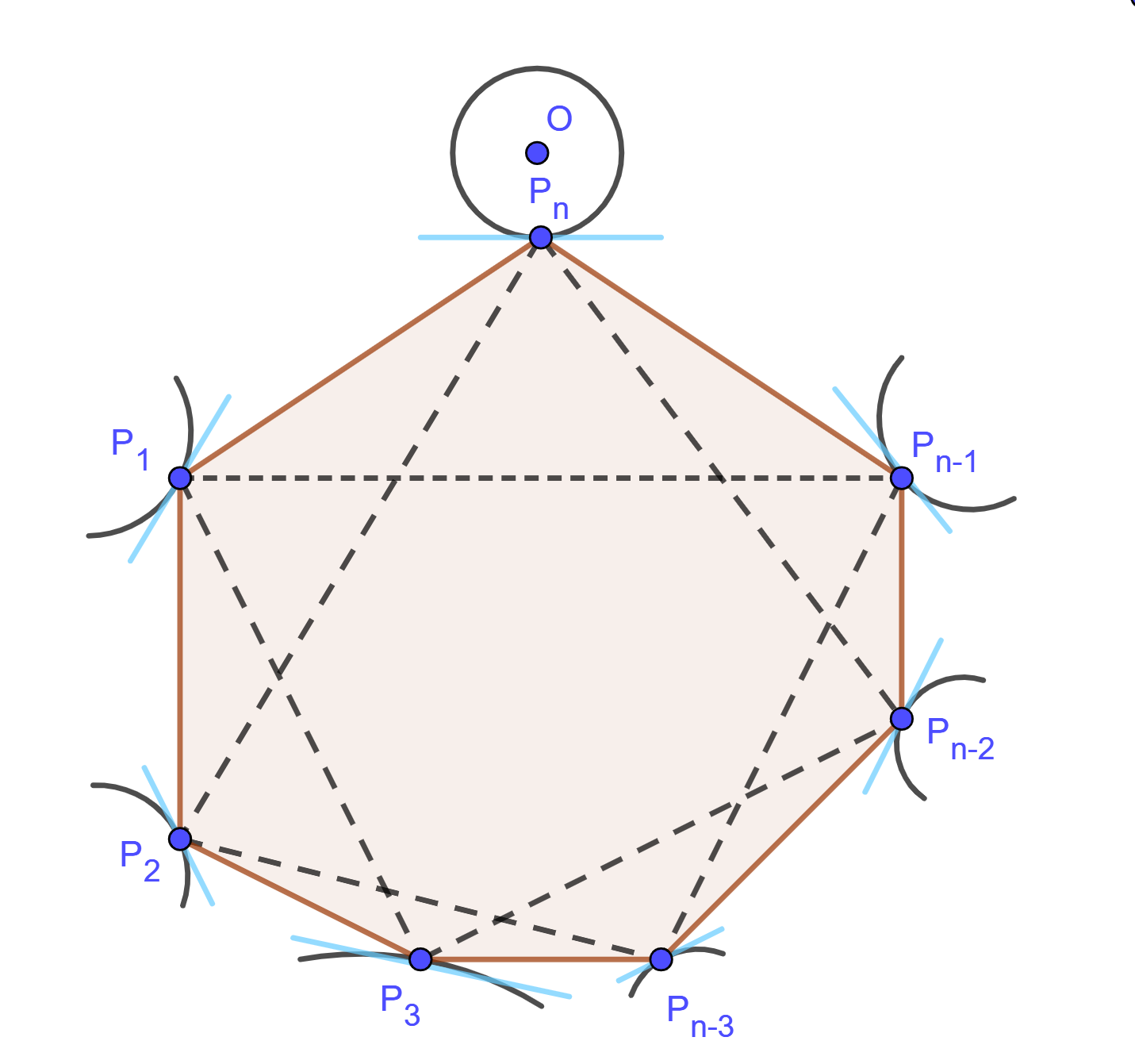}
\includegraphics[width=0.40\textwidth]{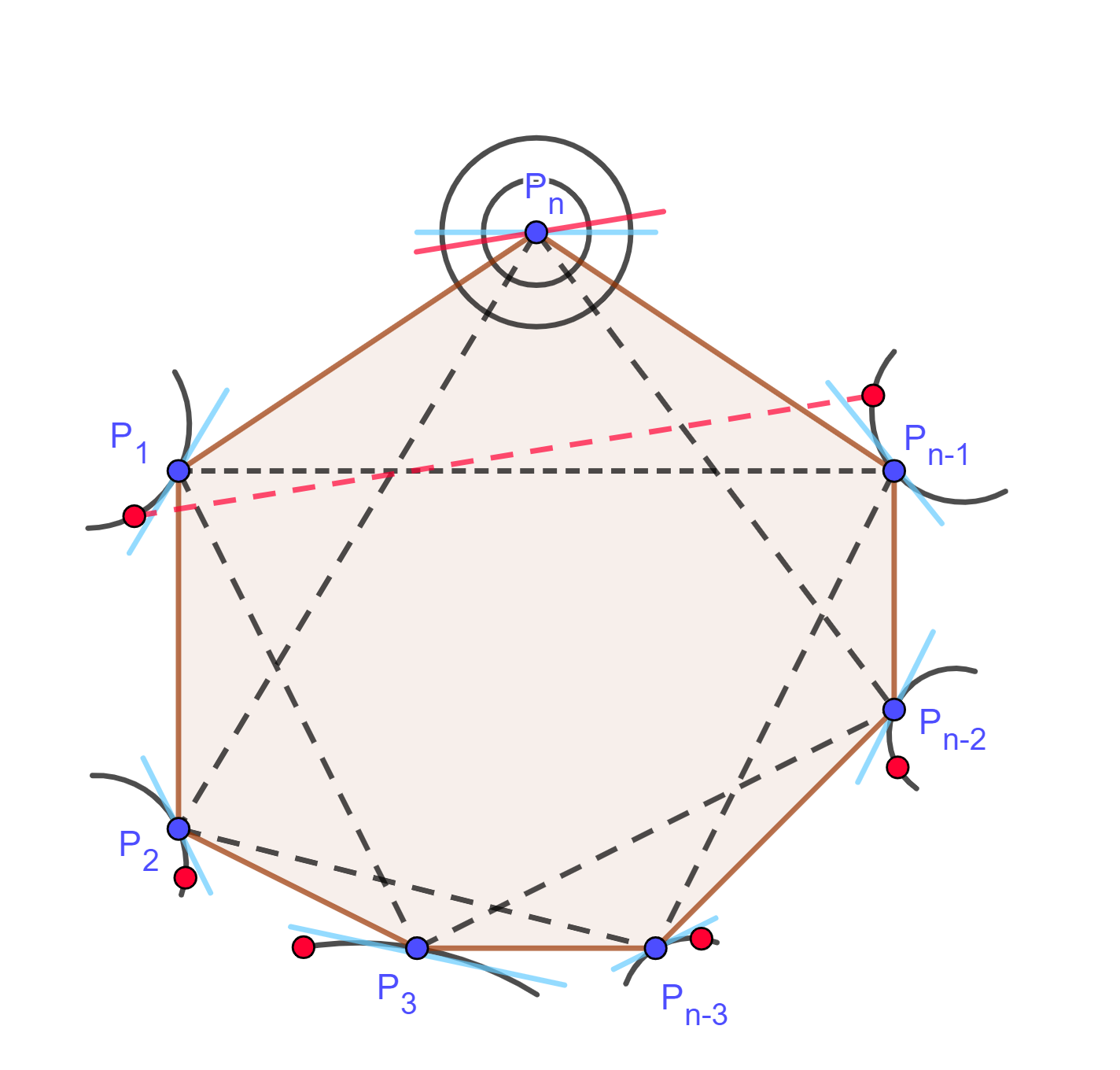}

	\caption{Growing circles}
	\label{fig:birth}
\end{figure}

\begin{remark}
{\em  How to determine the sign of $a_n$ in a geometric way ?}
Let $M_i$ be the center of curvature of $C_i $ at the point $P_i$  then $a_i > 0 $ if  $M_i$ is on the left side of the tangent line and $a_i < 0$ if $M_i$ is on the right side.
(the tangent line has orientation from $P_{i-1}P_{i+1}$).
\end{remark}

\begin{remark}
In case $a_n \det H_{n-1} \ne 0$ the formula (\ref{eq:detH}) shows that $\det H[r] = 0 $ has 2 different roots as soon as $K \ne 0$. It follows, that if $|r|$ grows we get another sign for $H[r]$, which has an effect on the Morse index.
\end{remark}

\noindent
We can extend the idea behind the proof to the growing of more points at the same moment.
\begin{example}
 We start with a polygon $\mathcal P = P_1\cdots P_n$ in `general position''.
The directions of the small diagonals determine potential tangent directions. Consider `reference' circles with centers $M_i$ and radius $r_i$;  each of them with a given sign of $a_i$ and tangent at $P_i$ to the tangent lines.
Next consider the circles with center $M_i[s] $ and radius $s r_i$, still tangent to the same tangent line which coincide for $r=1$ with our reference circle's.

\noindent
Our polygon is critical for every $s$. The matrix elements are  $a_i[s] = s a_i$  and $b_i[s] = s^2 b_i$, where the $a_i$ and $b_i$ are defined for the reference circles.

\noindent
It follows ($k=1,\cdots n$):
$$  \det H _k[s]=   s^k( a_1 a_2 \cdots a_k  + s^2 K_k [s]) $$
For small enough $s>0$  the sign is  given by  the sign of  $a_1 a_2 \cdots a_k$. With the help of the Sylvester rule one can compute the index of the critical polygon. By changing the signs  of $a_i$ (taking the reference circle at the other side of the tangent line) one changes the index. By repeating this procedure one can get any index.
\end{example}

\subsection{The birth of centered circles}
\medskip
The circles in Proposition \ref{p:birth1} don't have the center in $P_i$. The following statement tells about that situation (see the right hand side of Fig \ref{fig:birth}):
Let $C_1, \cdots , C_{n-1}$ be curves and $C_n$ is a point curve $P_n$  (all disjoint). Let the point $P_n$ grow to a small circle $C(P_n,r)$.
We look for the critical points of $\mathcal A$: It turns out that generically each critical point on $W = C_1 \times \cdots \times C_{n-1}$ generates two critical points on $W \times S^1$:

\begin{proposition}\label{p:birth2}
Given points $P_1,\cdots ,P_{n-1}$ on smooth curves  $C_1, \cdots ,C_{n-1}$ and a point $P_n$ such that   the polygon $P_1P_2 \cdots P_{n-1} P_n$ is a critical point of $\mathcal A$  on $W$.\\
Assume transversality: $ P_{i-1}P_{i+1} \pitchfork P_{i}P_{i+2} $ for $i= 1,\cdots,n$.
Then,   for $r$ small enough, there exists  near to $P_1P_2 \cdots P_{n-1}P_n$  on $W \times C(P_n,r)$ exactly two critical polygons    $P_1^{\pm} \cdots P_{n-1}^{\pm}P_n^{\pm}$, where $P_n^{\pm}$  on $C(P_n,r)$.\\
If the original critical point is  Morse of index $\mu$ then the two new critical points are again Morse and have index $\mu$, resp. $\mu + 1$.
\end{proposition}

\begin{proof}
We start with a billiard type construction.  Our reasoning  applies to local neighbourhoods of the points $P_1,\cdots,P_n$. 
The transversality  conditions imply that each line $P_{i-1}P_{i+1}$ intersects $C_{i-1}$, resp $C_{i+1}$ transversal at $P_{i-1}$, resp $P_{i+1}$, $ (i =2, \cdots, n-2)$,\\
Choose coordinates $t_1,t_2$  on $C_1,C_2$ such that $P_1$ and $P_2$ correspond to $t_i=0$. Next we define $t_3,\cdots, t_{n-2}$ such that 
 $$C_i(t_i) C(t_{i+2}) \parallel  \mathbb{R} C_{i+1}^{'}(t_{i+1})\; \;, \; (i=1, \cdots n-2.)  $$
 Due to the transversality conditions, this well defined in a neighbourhood of $(0,0)$.  The maps $(t_i,t_{i+1}) \rightarrow (t_{i+1},t_{i+2})$ are local diffeomorphisms by the same reason. 

We intend to use $(t_1,t_2)$ as a coordinate system near $P_n$. Consider the map 
$$ (t_1,t_2) \rightarrow Q (t_1,t_2) = \big{(} C_2(t_2) + \mathbb{R} C_1^{'}(t_1) \big{)} \cap \big{(} C_{n-2}(t_{n-2})+ \mathbb{R} C_{n-1}^{'}(t_{n-1}) \big{)}.   $$ 
Also this map is a local diffeomorphism, due to the transversality of the images of the two coordinate-axis, which intersect in $Q=P_n$. Let  $\Phi(X)=  (t_1(X),t_2(X))$ be its inverse.\\ 
Next parametrize $C(P_n,r)$ by $X(t) = P_n+ r(\cos t,\sin t)$ , $t \in [0,2 \pi]$.  While $X$ is moving around the circle, we  consider the 
argument $\beta_r(t)$  of  the chord from $C_{n-1}(t_{n-1}(X))$ to $ C_1(t_1(X))$.  Note that for $r$ small enough the image of $\beta_r$
is contained in an arbitrary small circle sector around the (limit) direction $\beta_0$.\\
In order to satisfy the  parallel condition between $C_1$ and $C_{n-1}$ we have to determine those points $X$ on $C(P_n,r)$ where the tangent line to the circle is parallel to the chord. This is given by the condition $t - \beta_r(t) = \pm \pi / 2$.
For $r$ small enough the graph of $t -\beta_r(t)$  is transversal to levels $\pm \pi /2$, since this is the case if $r=0$ and moreover the 2 points of intersection survive during the small deformation.
Since our constructing takes care of all other parallel conditions we have shown that the two resulting polygons are critical.

\smallskip
\noindent
The statement about the Morse indices in $P_n^{\pm}$ follows in the same way as in Proposition \ref{p:birth1}.  The entries in the formula (\ref{eq:detH}) now depend all on $r$, but for $r$ small enough $a_n$ and $\det H_{n-1}$  are bounded away from $0$.

\end{proof}

\section{Polygons on a single curve} \label{s:allonone}

\subsection{Local, global and zigzags}
The case of a single curve is of special interest.  In this case we meet also non-isolated singularities of the area function, due to coinciding vertices.  If all vertices are distinct, then $\mathcal A$  behaves  as if the vertices were on different local curves and we can use all the facts about these from the proceeding sections. We call these the {\em global case}, while coinciding vertices are related to local effects. 
 This can give rise to a zig-zag behaviour of critical polygons.
\begin{definition}
 \em Adding a zig-zag to a n-gon $\mathcal P = P_1P_2\cdots P_n$  is a the (n+2)-gon 
$\overline{\mathcal P } = P_1\cdots P_{i-1}P_{i}P_{i-1}P_iP_{i+1}\cdots P_n$
\end{definition}
Note that the $i^{th}$ condition for critical polygons has two aspects:
\begin{itemize}
\item $P_{i+1}=P_{i-1}$  or
\item $T(P_i) \parallel P_{i-1}P_{i+1}$
\end{itemize}
Therefore we have:
\begin{proposition}
Let $\mathcal P $ be a critical n-gon on $C$, then any  (n+2)-gon $\overline{\mathcal P }$ with a zigzag added to $\mathcal P $  is also critical on $C$.
\end{proposition}

\noindent
N.B.  We meet a similar behaviour in case of two curves $C_{even}, C_{odd}$, where the vertices are on the corresponding curve: $P_i \in C_{even}$ when $i$ is even, and  $P_i \in C_{odd}$ when $i$ is odd.

\smallskip
\noindent
The next statement works in the generic  case:

\begin{proposition}
If the critical polygon $\mathcal P$ is Morse and $C_{n-1}^{'} \times C_n^{'} \ne 0$ 
then $\overline{\mathcal P }$ is also critical and Morse  and: \hspace{0.3cm}
Morse-index $(\overline{\mathcal P} )= 1 +$  Morse-index  $({\mathcal P})$
\end{proposition}
\begin{proof}

We can assume that $i=n$, so $\overline{\mathcal P } = P_1P_2\cdots P_{n-1}P_nP_{n-1}P_n$ 
We compare the Hessian matrices: For   $n=5$ these are as follows:

$$
H =        \begin{pmatrix}
               a_1  &  b_1   &   0  &   0 &  b_5 \\
              b_1  &  a_2   &  b_2 & 0  &   0\\
						0 & b_2 & \ a_3& b_3  & 0\\
0 & 0 &  b_3 &  a_4 & b_4 \\
               b_5  & 0  &  0 &  b_4  & a_5
             \end{pmatrix}
\; \; ;  \; \;
\overline{H} =        \begin{pmatrix}
               a_1  &  b_1   &   0  &   0 & 0 & 0 & b_5 \\
              b_1  &  a_2   &  b_2 & 0  &   0 & 0 & 0 \\
						0 & b_2 & \ a_3 & b_3  & 0 & 0 & 0\\
0 & 0 &  b_3 &  a_4 & b_4 & 0 & 0\\
               0  & 0  &  0 &  b_4  & a_5 & -b_4 & 0 \\
0 & 0 & 0 & 0 & -b_4 & 0 & b_4\\
b_5 & 0 & 0 & 0 & 0 & b_4 & 0 
             \end{pmatrix}
$$
\noindent
Notice that we get two extra rows and columns.  The entries in the new row and columns are 0 on the main diagonal . Due to the zigzag  we have  that three  tangent vectors are the same or have opposite direction. As a consequence: \\ $\overline{b}_{n+2} = b_n \;  , \; \overline{b}_{n+1} = b_{n-1}\; ,\; \overline{b}_n = - b _{n-1}$, $a_{n+1} = a_{n+2} = 0.$

\medskip
\noindent
We apply Sylvester's rule for our Hessian $H$ and compare the sequence ( \ref{eq:sylv})
$$ 1, \det  H_1, \det H_2 , \cdots , \det H_{n-1} , \det H_n $$
with the corresponding sequence for $\overline{H}$:
$$1, \det\overline{  H}_1, \det\overline{ H}_2 , \cdots , \det \overline{H}_{n-1},  \det \overline{ H}_n ,  \det \overline{H}_{n+1} , \det \overline{H}_{n+2}  $$  
Due to our genericity assumption both sequences satisfy the  Sylvester assumptions.
Note that : $H_k = \overline{H}_{k}$ as soon as $k \le n-1$.
Moreover by elementary determinant operations:\\
$\det \overline{H}_{n+2} =  - b_{n-1}^2 \det H_n $,\\
$\det \overline{H}_{n+1} = -b_{n-1}^2 \det H_{n-1} .$\\
Let $\epsilon_k$ be the sign of $\det H_k$.
The sign sequences of the two determinant sequence above are as follows:\\
$ + , \epsilon_1, \epsilon_2,\cdots, \epsilon_{n-1}, \epsilon_n $,\\
$ + , \epsilon_1,\epsilon_2,\cdots, \epsilon_{n-1}, \rho \; , - \epsilon_{n-1}, - \epsilon_n $.\\
where $\rho$ is the sign of $\overline{H}_n$.  It is clear that independent of the value $\rho$ the number of sign changes in the second sequence is one more than in the first.
\end{proof} 
\noindent
In the case of even $n$ one can  meet so-called zigzag-trains (as in the circle case, discussed in  (\cite{Si-circle})): Start with $P_1$ and $P_2$: construct $P_3$  by the 
parallel criterion, and continue in this way: $P_4$, etc.  For some $k$ switch to the condition $P_{k+1}= P_{k-1}$ and continue with 
$P_{k+2} =  P_{k-2}$ until  we arrive in $P_1$.  One can also put some zig-zags in between (does not matter where).
By moving $P_1$ and $P_2$ one gets 2-dimensional families of polygons: {\it zigzag trains.}

\smallskip
\noindent
Special zigzag-trains arise from  two different  points on the curve. The  2-gon $P_1P_2$ is always critical. Adding zigzags give critical 4-gons $P_1P_2P_1P_2$, etc.    a series of non-isolated critical  polygons. Also the case when all points coincide is a non-isolated  critical polygon. So there are plenty of non-isolated critical polygons!  Their
(Bott)-Morse theory can become very complicated.

\subsection{Polygons in a circle or ellipse}
In \cite{Si-circle}
we give a complete description of all critical polygons and indices. 
The main theorem  gives geometric criteria for the critical points and determines also the Hesse matrix at those points. Most of the critical points are of Morse type and look as a regular star, but several  of them have zigzag behaviour. The Morse index is determined  by  combinatorial data.  We give a summarized version where $\alpha_i = \angle P_i M P_{i+1}$ and $M$ is the center of the circle. 

\begin{theorem}\label{t:mainthm}

The signed area function for polygons on a circle (defined on the reduced configuration space) has critical points iff all $|\alpha_i |$ are equal. 
These critical points are 
isolated or (if the number of vertices $n =$ even) contain also a 1-dimensional singular set. More precise
\begin{itemize}
\item[1.] The isolated singularity types are regular stars, zigzag stars and if $n=$odd also degenerate stars,
\item[2.] All regular and zigzag stars  are Morse critical points,
\item[3.] Degenerate stars are  degenerate isolated  critical points if $n$ is odd.  
\item[4.] The non-isolated case only occurs if  n = even and 
includes the complete fold, zigzag trains and degenerate stars.
The non-isolated part of the critical set contains $\tbinom{n}{\frac{n}{2}}$ branches, which meet only  at the complete fold and the degenerate stars.
\end{itemize}
\end{theorem}
\begin{figure}[ht]
	\centering
	\includegraphics[width=1.05 \textwidth]{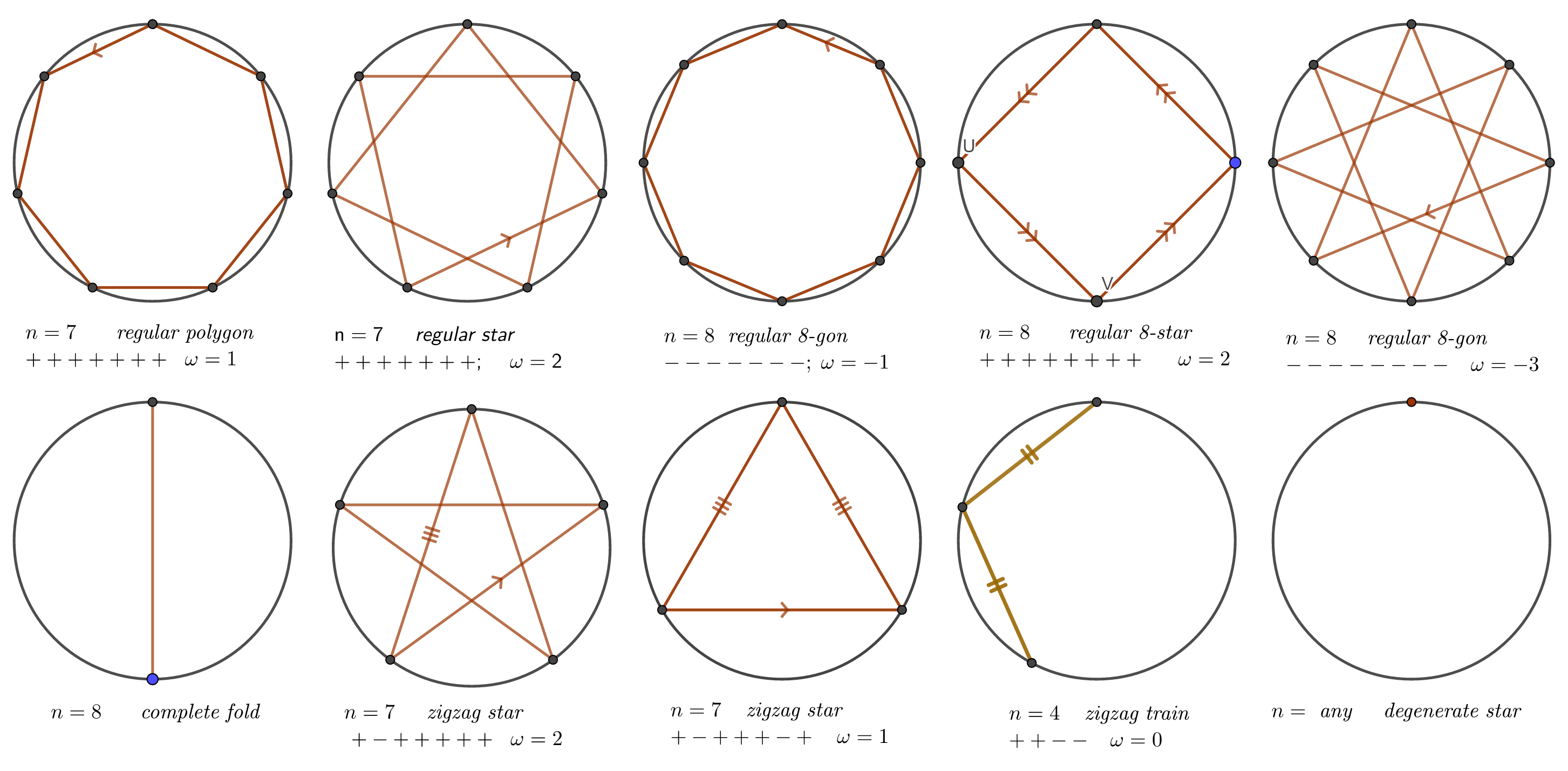}
	\caption{Some critical configurations}
	\label{fig:stars}
\end{figure}

\noindent
We computed also the index of the gradient vector field at the degenerate star by Euler-characteristic arguments. In section 3  of \cite{Si-circle} we discussed  the Eisenbud-Levine-Khimshiashvili method to calculate this index. This related nicely to a combinatorial question, which is solved in \cite{vdKS}.

\smallskip
\noindent
Note, that the problem of extremal area polygons in an ellipse is also solved due to the existence of an area preserving affine map.

\section{Piecewise differentiable curves}\label{s:piecewise}
In many situations piecewise smooth curves occur. These are differentiable curves with finitely many  (break)  points, where only the right-derivative and the left
derivative exist, $C_{i,-}^{'}$, resp.  $C_{i,+}^{'}$. We denote the corresponding tangent vectors by $T_{i,-}$, resp $T_{i,+}$.

\smallskip
\noindent
We can  determine the critical points  of  $\mathcal A$ with the help of generalized derivatives, e.g. the Clark subdifferential. 
The generalized derivative of $C(t)$ at a break point is given by    $\delta C =  ch(T_{-},T_{+})$,  the  convex hull of the right and the left tangent vector; the corresponding cone in the tangent spaces is denote by $TC$.
We avoid $T_{-} + T_{+} = 0$  !

\smallskip
\noindent
The area function $\mathcal{A}$ is an example of a `continuous selection'. Its critical point and Morse theory are especially studied in \cite{JP} and \cite{APS}. A continuous function $f$ is called a continuous  selection of functions
$f_1,\cdots,f_m$ if $I(x) =\{i \in \{1,\cdots,m\} | f_i(x) = f(x)\}$ is non-void . The set $I(x)$ is called the active index set of $f$ at the point $f$.

\smallskip
\noindent
 If all the functions $f_i$ are smooth ($C^1$ ) then  $f$ is locally Lipschitz continuous and the Clark subdifferental of $f$ is given by 
$$ \delta f (x) = ch \{\nabla  f_i (x) | i \in \hat{I}(x\},$$ where $\hat{I}(x) =\{ i | x \in cl  \; int \{x |f(x) = f_i (x) \}\} $. \\
Subdifferentials satisfy the usual calculus rules: vectors replaced by sets.    

\smallskip
\noindent
A  point $x_0$ is called a critical point of a locally Lipschitz  continuous function iff $O \in \delta f (x_0)$.\\
 Locally Lipschitz continuous functions satisfy the first Morse lemma:  No critical points imply a (topological) product structure.
We apply this to $\mathcal A$:

\begin{theorem} \label{t:ns-crit_va}
Let $C_1,\cdots, C_n$ be piecewise smooth  curves in the plane. $\mathcal A$   has a critical point at the polygon $P_1 \cdots P_n$ iff 
$O \in \delta C_i \times ( C_{i+1} - C_{i-1}) $ for all $i$   at $(t_1,\cdots,t_n.$  with $P_i = C(t_i)$\\
This means:

 \begin{itemize}
\item $P_{i+1}=P_{i-1}$  or
\item $ P_{i-1}P_{i+1} \in   TC(P_i)$
\end{itemize}
\end{theorem}
\noindent
The paralell condition is now replaced by the tangent cone condition (Figure \ref{fig:NDpic}).
\begin{figure}[h]
	\centering
		\includegraphics[width=0.85\textwidth]{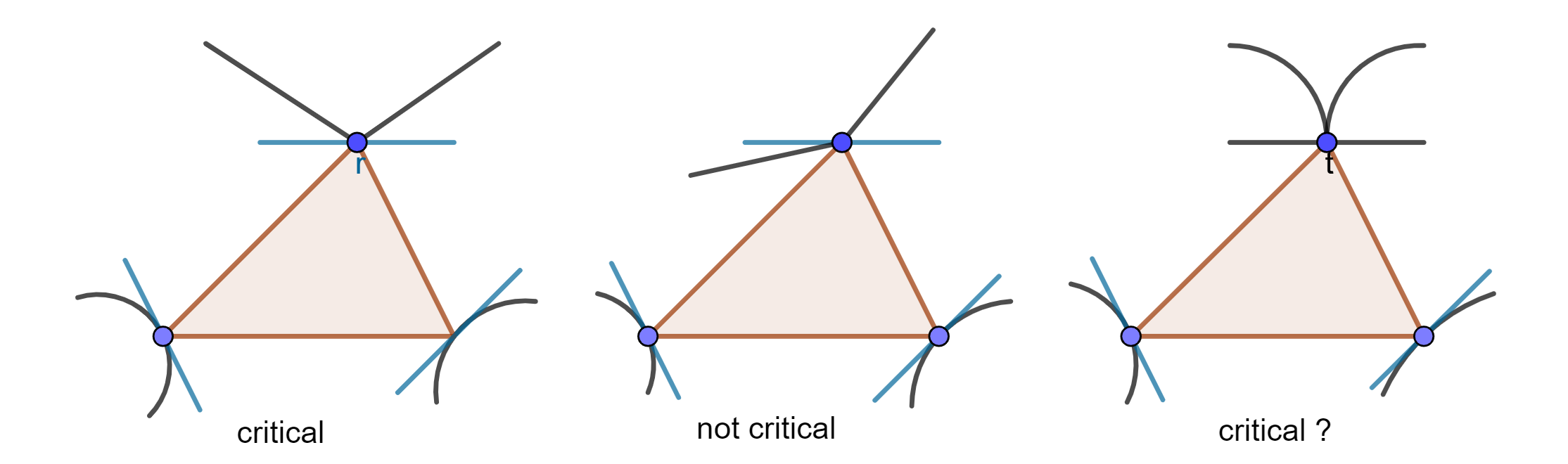}
		\caption{Non smooth critical points.}
	\label{fig:NDpic}
\end{figure}
\noindent
We don't  treat the Morse theory, we restrict ourselves to the following remarks:
Morse theory for continuous selections is developed in \cite{APS}, but in the case of the area function $\mathcal A$  an extension seems to be necessary.\\
A second approach can be sketched as follows.  Use a rounding off curve $\tilde{C}_i$ of $C_i$ in a very small neighbourhood of the breakpoints. It is clear that the critical points of $\mathcal A$  in the two situations are in 1-1 correspondence. We expect even that $\mathcal A$ in the two cases is topologically equivalent. Next one can use smooth Morse theory to determine the type of the critical points.  
We leave this idea for further studies. It seems  interesting in the case that each $C_i$ is a polygon, especially in the case of coinciding curves.

\smallskip
\noindent
The triangle case is intensively studied in computational geometry. Mostly to invent algorithms to select the maximal area triangle in a polygon with many vertices. It could be of interest
to study the critical point theory of triangles, 4-gons and higher.  One can also meet non-isolated singularities.  Is it possible to use the simplicial structure and discrete Morse theory ?

\section{Tangential sliding}\label{s:tsliding}
\subsection{Critical Points}

We us the notations $C_1,\cdots,C_n$ for the curves, which we give a direction and a parametrization. 
\begin{figure}[ht]
	\centering
		\includegraphics[width=0.40\textwidth]{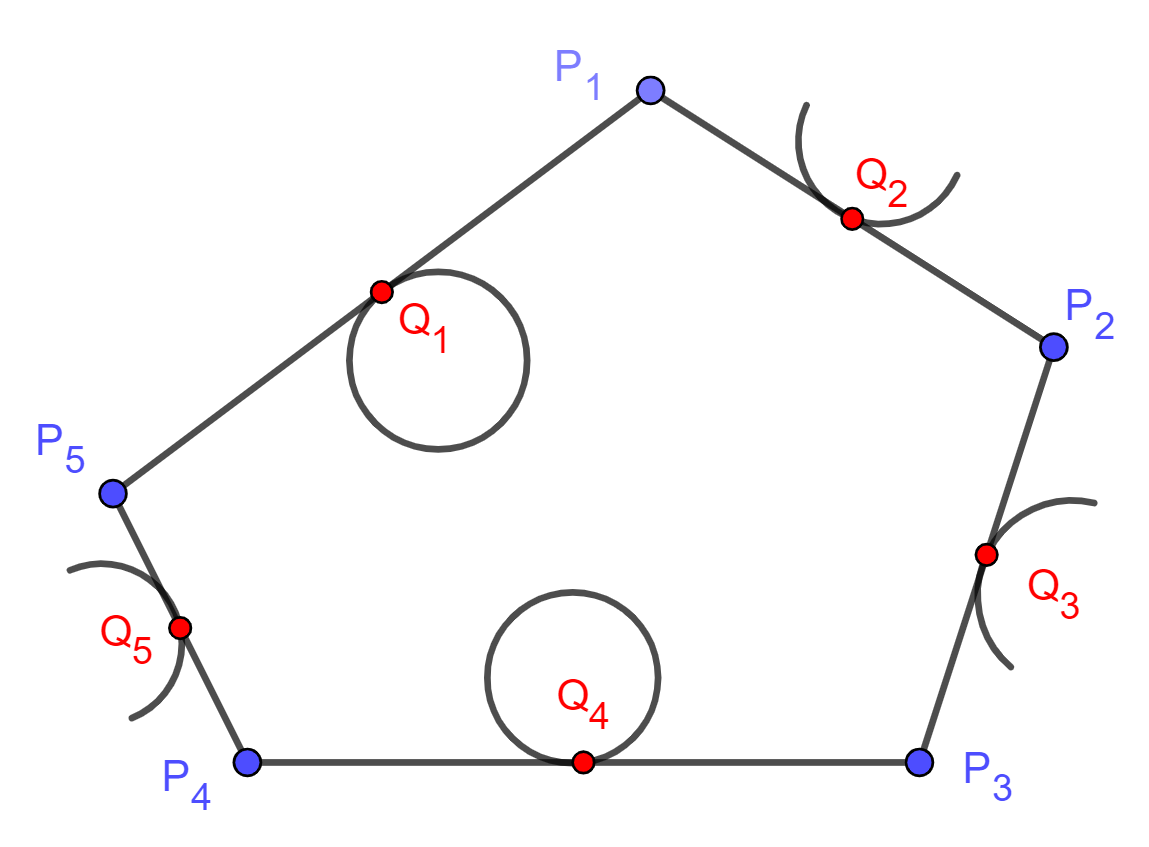}
	\caption{Tangential polygon  in midpoint position}
	\label{fig:slidingArea}
\end{figure}
On each of the curves we consider a point $Q_i$. The tangent lines in $Q_i$ define a polygon, by taking the intersection points   $P_i$ between the tangent lines in $Q_{i}$ and $Q_{i+1}$. The signed area of $P_1P_2\cdots P_n$ defines a function $ t \mathcal{A}: (S^1)^n \to  \mathbb{R}$ . The point $Q_i$ is not defined if the two tangent lines are parallel. One could probably  add the values $\pm \infty$ to the source space.

\begin{theorem}\label{t:crit_ta}
Critical points of  $t \mathcal{A}$ are polygons where the vertices are midpoints or points with vanishing curvature.
\end{theorem}
\begin{proof}
$t \mathcal{A}$ depends on $(t_1,\cdots,t_n)$. Fix next all $t_k$ with $k\ne i$ and compute the partial derivative with respect to $t_i$.  It is sufficient to consider the triangle $Q_{i-1}Q_iQ_{i+1}$. The statement for triangles is folklore (see \ref{ss:triangle}) and follows by elementary computations.  
\end{proof}

\subsection{Triangle case}\label{ss:triangle}
More than 100 years ago E.B. Wilson \cite{wilson} showed, that for triangles on convex curves vertex area $\mathcal{A}$ and tangential area $t \mathcal{A}$ have the same critical points. He used an infinitesimal proof and asked the question: {\it Is there any `easy' way of reaching this result by exclusively analytic methods now in vogue ?}. This follows now anyhow from our Theorem \ref{t:crit_va} and Theorem \ref{t:crit_ta}.
By elementary geometry the midpoint condition for the tangential triangle and the parallel condition are equivalent.

\smallskip
\noindent
If $n >3$ there is no longer the coincidence of critical points for both type of slidings.

\subsection{Related work}

In the paper \cite{cmd}
one considers given angles of the polygon and give geometric conditions for extremal perimeter and area.
The paper \cite{dT} contains not only criteria for extermal perimeter, but at the end also for area.
The midpoint condition is contributed  to M.M. Day (1974).

\section{Towards an Inner Area Billiard}\label{s:billiard}
The critical polygon construction for the area function $\mathcal A$ can be used to define a new billiard. The approach will be similar to the constructions of (usual) billiard from the perimeter function 
$$\mathcal{P}er = |P_1P_2| + |P_2P_3| +  \cdots |P_{n-1}P_n| + P_nP_1|$$ 
and the outer billiard as explained below. We describe both in cases of a differentiable strict convex curve $C$. As references to billiards  we give \cite{Ta} and \cite{GT}.

\subsection{(Inner) Perimeter Billiard}
For polygons on $C$ the critical points of $\mathcal{P}er$ are determined by the reflection law: Two consecutive edges reflect in the tangent line at the common vertex. One can use the same rule for construction of the billiard. Start with $P_1 \ne P_2$ on $C$, determine $P_3$ via the reflection rule in $P_2$ as intersection of the reflected ray with $C$, etc. The closed orbits correspond to the critical points of $\mathcal{P}er$.  To distinguish from other billiards  we will call this the Inner Perimeter Billiard.

\subsection{(Outer)Area Billiard}
Next we consider polygons where the edges are tangent to the curve $C$. The critical points of $t\mathcal A$  are determined by the mid-point property: Any edge is tangent to $C$ at its mid-point (Theorem \ref{t:crit_ta}). The Outer Area Billiard is defined by that  rule : Start with any point $P_1$ outside the convex region , draw a tangent line to $C$ (there are 2 choices) and take the point $P_2$ on the tangent line such that the point of tangency is the mid-point. Construct $P_3$ via  the (other) tangent line to $C$ and $\cdots$, etc. The closed orbits correspond to the critical points of (outer) area $t \mathcal{A}$.
\subsection{Inner Area Billiard}
\begin{figure}[h]
	\centering
		\includegraphics[width=0.35\textwidth]{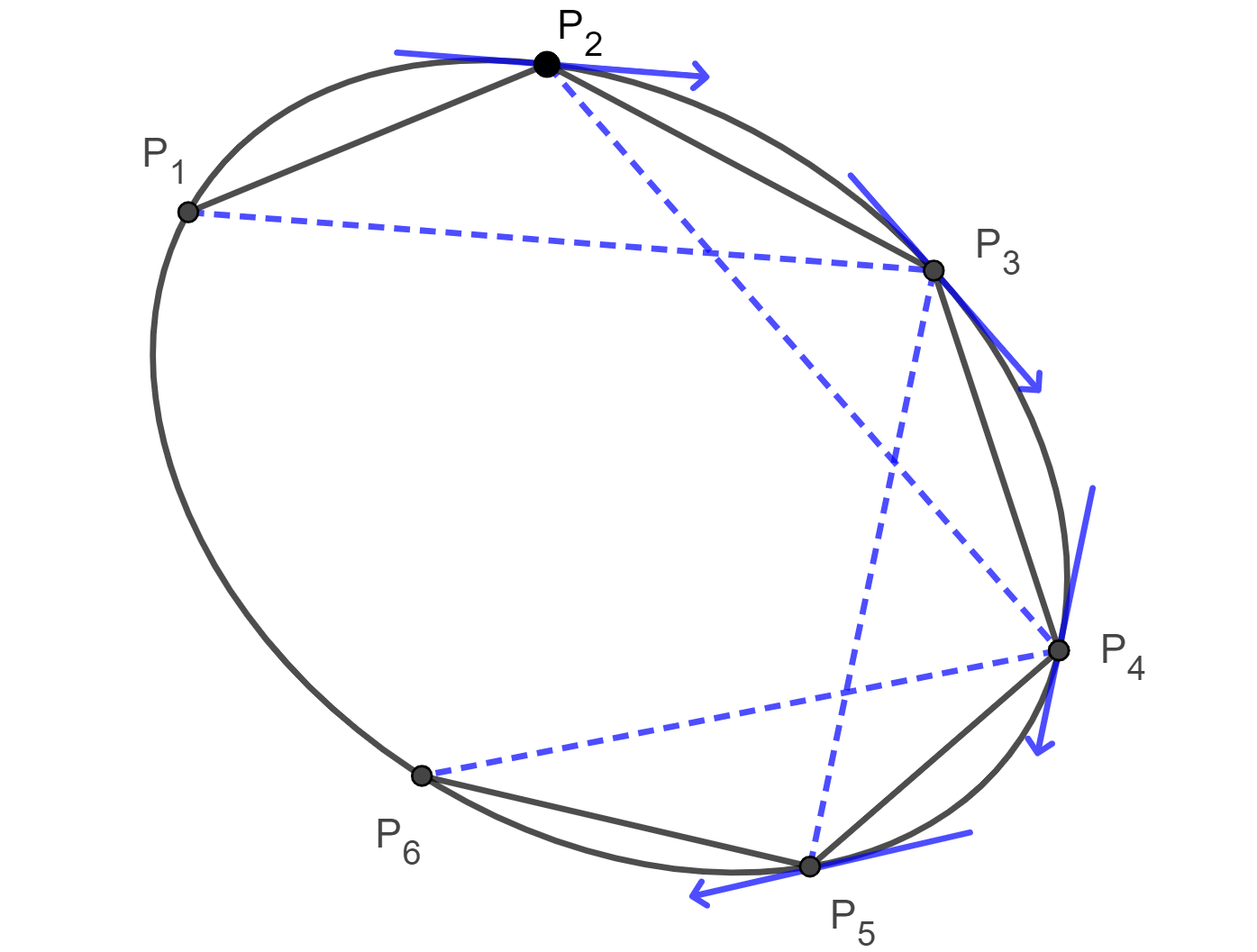}
	\caption{The inner area billiard.}
	\label{fig:billiard}
\end{figure}
Now we describe our new billiard by using the (inner) area function $\mathcal A$. Start with a polygon inscribed in a convex curve. The critical points of $\mathcal A$ are given by the parallel rule:\\
$T(P_i) \parallel  P_{i-1}P_{i+1} \; \; , i=1,\cdots n$. We exclude the zigzag-rule. 

\smallskip
\noindent
Start with $P_1 \ne P_2$ on the curve and construct $P_3$  by intersecting  the line through $P_1$  parallel to $T(P_2)$ with $C$, construct $P_4$ by intersecting the line through $P_2$ parallel to $T(P_3)$, etc. The closed orbits are the critical polygons of (inner) area $\mathcal A$.  We call this billiard the Inner Area Billiard.  

\smallskip
\noindent
It looks interesting to study the properties of this billiard in detail. Questions are:
\begin{itemize}
\item Do caustics exist ?
\item Existence of closed n-orbits with given winding number
\item Other questions in ordinary billiard theory
\end{itemize}
Note that the area function on the ellipse has the property that it has a caustic which is again an ellipse. Each critical polygon is  non-isolated.  The types of critical  orbits follow from \cite{Si-circle}. The caustic exists and is also an ellipse.


\end{document}